\documentclass[11pt]{amsart}
\usepackage{amsmath,amssymb,latexsym,soul,cite,mathrsfs}
\usepackage{mathtools}
\usepackage{color,enumitem,graphicx}
\usepackage[colorlinks=true,urlcolor=blue,
citecolor=red,linkcolor=blue,linktocpage,pdfpagelabels,
bookmarksnumbered,bookmarksopen]{hyperref}
\usepackage[english]{babel}
\usepackage[left=2.9cm,right=2.9cm,top=2.8cm,bottom=2.8cm]{geometry}
\usepackage[hyperpageref]{backref}

\usepackage[colorinlistoftodos]{todonotes}

\makeatletter
\providecommand\@dotsep{5}
\def\listtodoname{List of Todos}
\def\listoftodos{\@starttoc{tdo}\listtodoname}
\makeatother

\numberwithin{equation}{section}

\newtheorem{theo}{Theorem}[section]
\newtheorem{lemma}[theo]{Lemma}
\newtheorem{prop}[theo]{Proposition}

\newtheorem{remark}[theo]{Remark}

\makeatletter
\renewenvironment{proof}[1][\proofname]{
  \par\pushQED{\qed}\normalfont
  \topsep6\p@\@plus6\p@\relax
  \trivlist\item[\hskip\labelsep\bfseries#1\@addpunct{.}]
  \ignorespaces
}{
  \popQED\endtrivlist\@endpefalse
}
\makeatother

\providecommand{\abs}[1]{\lvert#1\rvert}
\providecommand{\norm}[1]{\lVert#1\rVert}

\DeclareMathOperator{\cat}{cat}

\title[Positive solutions to SP system in expanding domains]{Small normalised solutions for a\\
 Schr\"odinger-Poisson system in expanding domains:\\
multiplicity and asymptotic behaviour}

\author[E. G. Murcia]{ Edwin Gonzalo Murcia}
\author[G. Siciliano]{Gaetano Siciliano}
\address[E. G. Murcia]{\newline\indent Departamento de Matem\'aticas
\newline\indent 
Facultad de Ciencias
\newline\indent 
 Pontificia Universidad Javeriana 
\newline\indent 
Carrera 7 No. 43-82 Bogot\'a, Colombia}
\email{\href{mailto:murciae@javeriana.edu.co}{murciae@javeriana.edu.co}}
\address[G. Siciliano ]{\newline\indent Dipartimento di Matematica
\newline\indent 
Universit\`a degli Studi di Bari Aldo Moro
\newline\indent 
Via E. Orabona 4, 70125 Bari, Italy}
\email{\href{mailto:gaetano.siciliano@uniba.it}{gaetano.siciliano@uniba.it}}



\begin{document}

\begin{abstract}
Given a smooth bounded domain $\Omega\subset \mathbb R^3$,
we consider the following  nonlinear Schr\"odinger-Poisson type system
\begin{equation*}
 \left\{
	    \begin{array}{ll}
	-\Delta u+ \phi u -\abs{u}^{p-2}u =  \omega u & \quad \text{in } \lambda\Omega,\\
	-\Delta\phi =u^{2}& \quad \text{in }\lambda\Omega,\\
	u>0 &\quad \text{in }\lambda\Omega,\\
	u =\phi=0 &\quad \text{on }\partial (\lambda\Omega),\\
	\int_{\lambda\Omega}u^{2} \,\text{d} x=\rho^2
	    \end{array}
    \right.
\end{equation*}
in the expanding domain $\lambda\Omega\subset \mathbb R^{3}, \lambda>1$
and $p\in (2,3)$, in the unknowns $(u,\phi,\omega)$.
 We show that, for arbitrary large values of the expanding parameter $\lambda$
and arbitrary  small values of the mass $\rho>0$, the number of solutions 
is at least the Ljusternick-Schnirelmann category
of $\lambda\Omega$.
Moreover we show that as $\lambda\to+\infty$ the solutions found converge to
a ground state of the problem in the whole space $\mathbb R^{3}$.
\end{abstract}
\maketitle

\begin{center}
	\begin{minipage}{8.5cm}
		\small
		\tableofcontents
	\end{minipage}
\end{center}

\bigskip

\section{Introduction}

 Elliptic systems involving the   Schr\"odinger and the Maxwell equations have attracted a lot of interest in  mathematical physics
 in the last decades. Many authors have studied this intriguing problem, which takes into account the interaction of a nonrelativistic 
 particle with its own electromagnetic field and it is almost impossible to give a complete list of references on the topic.
 The huge literature existing  deals mainly with the equations settled   in the whole space $\mathbb R^{N},$ while only few works study the case of bounded domains where boundary conditions may even prevent the existence of solutions
 (see the work cited below).
We  cite here the pioneering paper of Benci and Fortunato \cite{BF} for two reasons:
first, it is the inspiration of our paper, and second because it seems  they first gave
a deduction of the equations which describe the interaction of the matter field 
with the electromagnetic field in the framework of Abelian Gauge Theories, in place
of the usual and classical Hartree and Thomas-Fermi-von Weizs\"acker Theory of atoms and molecules (see e.g.\cite{B,CL}).

Without entering in details in the physical and mathematical derivation of the equations
(beside \cite{BF}, the  interested reader is also referred to \cite{AP, DM, dA}) the search of stationary solutions
$$\psi(x,t)=u(x) e ^{i\omega t}\in \mathbb C\,, \  \quad u(x),\ \omega\in \mathbb R, $$
of the so called Schr\"odinger-Maxwell  system  in the purely electrostatic case, i.e.
$$\phi(x,t)=\phi(x), \quad \mathbf A(x,t)=\mathbf 0,$$ 
(where $(\phi, \mathbf A)$ is the gauge potential of the electromagnetic field)
 leads to the following model problem, known as the {\sl Schr\"odinger-Poisson system}:
\begin{equation*}
    \left\{
	    \begin{array}{ll}
	    	-\Delta u+ \phi u-|u|^{p-2}u= \omega u & \quad \text{in }\Omega,\\
	        -\Delta \phi =u^2 &\quad \text{in }\Omega,\\
	        u,\phi =0 &\quad \text{on }\partial \Omega.
	    \end{array}
    \right.
\end{equation*}
It is  assumed that the particle  is going to be ``observed'' in a  region $\Omega$
in $\mathbb R^{3}$ where it is  confined, and $p$ is a suitable exponent.
Here $u$  and $\phi$ are unknowns of the problem.

For what concern the number $\omega$, the frequency of the wave function, there are two different points of view,
depending on the problem one wants to study:
\begin{itemize}
\item[I.] $\omega$ is a given datum of the problem; this means that one is interested in finding wave functions with a given frequency. Sometimes this produces restrictions on the values of $\omega$ in order to obtain solutions. 
For the  problem in a bounded domain the reader is referred, for example, to\cite{Alves,PS,RS,S}. 
\item[II.] $\omega$ is not given; in this case the wave function is completely unknown,
no a priori value of the frequency is specified so the unknowns of the problem are $u,\phi$ and $\omega$.
In contrast to the previous case, this arises when  the $L^{2}$-norm of the solutions is a priori fixed. 
Beside the paper \cite{BF}, we cite \cite{CCM, TMNA} where a Neumann condition on $\phi$ is considered.
The problem in the whole $\mathbb R^{N}$ is studied in \cite{BS1,BS2} (see also the references therein) where
the main difficult was to recover some compactness {\sl \`a la Lions}.

\end{itemize}
We mention also \cite{ORL} where both cases I. and II. are treated.

\subsection{The problem addressed in this paper}

In this paper we are interested in the second point of view described above since
we believe it is more natural than the first. 
Indeed the wave function is usually
unknown, so also $\omega$ has to be treated as an unknown. This has the following consequences:
since the system is variational, we will find the solutions as critical points of
the associated {\sl energy functional} and the fact that $\omega$ is unknown led us to
find  critical points restricted to the constraint of functions with fixed $L^{2}$-norm.
Hence  $\omega$ will appear naturally 
as a Lagrange multiplier.
The restriction to functions $u$ with fixed $L^{2}$-norm has also a physical consistency since 
$|\psi(x,t)|^{2} = u^{2}(x)$ and 
it is known that the $L^{2}$-norm of the solutions of the Schr\"odinger equation is constant in time.

More specifically, our aim in this paper is to show the existence of solutions $(u,\phi,\omega)$ for the following system in an expanding domain $\lambda \Omega, \lambda>1$,
\begin{equation*}\label{omegal}\tag{$P_{\lambda}$}
    \left\{
	    \begin{array}{ll}
	    	-\Delta u+ \phi u-\abs{u}^{p-2}u =  \omega u & \quad \text{in }\lambda\Omega,\smallskip\\
	        -\Delta \phi =u^2 &\quad \text{in }\lambda\Omega,\smallskip\\
	        u>0 &\quad \text{in }\lambda\Omega, \smallskip\\
	        u,\phi =0 &\quad \text{on }\partial (\lambda\Omega),\smallskip\\
\displaystyle	        \int_{\lambda\Omega} u^{2} \text{d} x =\rho^2.
	    \end{array}
    \right.
\end{equation*}

The set $\Omega\subset \mathbb R^{3}$ is a smooth and bounded {\sl domain}, i.e., open and connected.
The requirement   $u>0$ is very natural, being $u$ the modulus of the wave function
(observe {\sl en passant} that the positivity of $\phi$ in $\lambda \Omega$ is granted for free).

It is classical by now (see the approach in \cite{BF})
that \eqref{omegal}  can be reduced to a single equation involving a nonlocal term;
indeed calling $\phi_{u}\in H^{1}_{0}(\Omega)$ the unique (and positive) solution of 
\[
    -\Delta \phi =u^2 \quad \text{in }\lambda\Omega,\quad \phi =0 \quad \text{on }\partial (\lambda\Omega)
\]
for fixed $u\in H^{1}_{0}(\lambda\Omega)$,
the problem \eqref{omegal} can be equivalently written as
\begin{equation}\label{eq:equacion}
 \left\{
	    \begin{array}{ll}
	-\Delta u+ \phi_{u} u-\abs{u}^{p-2}u =  \omega u & \quad \text{in }\lambda\Omega,\smallskip\\
	u>0 &\quad \text{in }\lambda\Omega,\smallskip\\
	u =0 &\quad \text{on }\partial (\lambda\Omega),\smallskip\\
\displaystyle		        \int_{\lambda\Omega} u^{2} \text{d} x =\rho^2,
	    \end{array}
    \right.
\end{equation}
to which we will refer from now on.
In this way by a (weak) solution of \eqref{eq:equacion} we simply mean a pair  $(u,\omega)\in H^{1}_{0}(\lambda\Omega)\times\mathbb R$ with $u$ having (squared) $L^{2}$-norm equals to $\rho^2$ for $\rho >0$, such that
\[
\forall v\in H^{1}_{0}(\lambda\Omega) : \int_{\lambda\Omega}\nabla u\nabla v\,\text{d} x+\int_{\lambda\Omega}
\phi_{u} uv\,\text{d} x-\int_{\lambda\Omega} |u|^{p-2}uv \,\text{d} x= \omega\int_{\lambda\Omega} uv\,\text{d} x.
\]
Of course, the solutions will depend on $\rho$ and $\lambda$. By standard variational principles we know that the solutions can be found as critical points of the $C^{1}$ energy functional (it is useful to have the explicit dependence on the domain considered)
\[
    I(u;\lambda\Omega):=\frac{1}{2}\int_{\lambda\Omega}\abs{\nabla u}^2\text{d} x+\frac{1}{4}\int_{\lambda\Omega}\phi_uu^2\text{d} x-\frac{1}{p}\int_{\lambda\Omega}\abs{u}^p\text{d} x
\]
on the constraint given by the $L^{2}$-sphere
\[
    \mathcal M_{\rho} (\lambda\Omega):=\Big\{u\in H^{1}_{0}(\lambda\Omega) : \int_{\lambda\Omega}u^{2} \text{d} x=\rho^2 \Big\}
\]
and $\omega$ appears as the Lagrange multiplier.
The term {\sl ground state} is used to refer to the solution with minimal energy.

To state our result, let us recall that in \cite{BS1} it has been  proved  that
there exists $\rho_1=\rho_1(p) >0$ such that for any $\rho\in(0,\rho_{1})$
the problem on the whole space $\mathbb R^{3}$ has a ground state solution $\mathfrak w_{\infty}$ and 
$$c_{\infty}:=\min_{u\in \mathcal M_{\rho}(\mathbb R^{3})} I(\mathfrak w_{\infty}; \mathbb R^{3})<0.$$
Moreover,  by the results in \cite{GPV}, the associated Lagrange multiplier $\omega_{\infty}$ is negative
and  the ground state $\mathfrak w_{\infty}$ is positive and radially symmetric.

\medskip 

Roughly speaking, our result given below states that for small value of the $L^{2}$-norm, the number of solutions is   influenced by the topology of the domain,
at least when the domain is very large.

Hereafter for a topological pair $Y\subset X$, $\cat_{X}(Y)\in \mathbb N$ denotes the Ljusternick-Schnirelmann category; if $X=Y$ we simply write $\cat X$.

The result is the following
\begin{theo}\label{th:main}
Let $p\in (2,3)$ and let $N=\emph{cat}\, \Omega$. There exists $\rho_{1}=\rho_{1}(p)>0$ such that, for every  $\rho\in(0,\rho_{1})$ there is 
 $\Lambda>1$ such that for any $\lambda \in (\Lambda,+\infty)$, the problem \eqref{eq:equacion} has at least 
 $N$ solutions $(u^i_{\rho,\lambda},\omega^i_{\rho,\lambda})\in \mathcal M_{\rho}(\lambda \Omega)\times \mathbb R$, $i=1,\ldots, N$. Moreover, for every $i$, as $\lambda \to +\infty$,
\[
    I(u^i_{\rho,\lambda};\lambda \Omega)\to c_\infty<0,\quad \omega^{i}_{\rho,\lambda}\to \omega_\infty<0,
\]
and, up to translations,
\[
    u^{i}_{\rho,\lambda}\to \mathfrak w_\infty \text{ in } H^1(\mathbb R^3).
\]
  
 Furthermore, if $\Omega$ is not contractible in itself, 
 besides the solutions just found, there is another one 
 $(\widetilde u_{\rho,\lambda}, \widetilde \omega_{\rho, \lambda})$ with $\widetilde u_{\rho,\lambda}$ nonnegative
and at a higher energy level. 
\end{theo}

We remark here, once for all,
that the solutions $u$ we will find are indeed solutions in the classical sense.

\medskip

\subsection{Comparison with known results}

Now an explicit comparison with the paper of Benci and Fortunato \cite{BF} is in order.
In the paper \cite{BF} infinitely many solutions $(u_{k}, \omega_{k})_{k\in \mathbb N}$ are found for the problem
\begin{equation}\label{eq:equacionBF}
 \left\{
	    \begin{array}{ll}
	-\Delta u+ \phi_{u} u =  \omega u & \quad \text{in }\Omega,\smallskip\\
	u =0 &\quad \text{on }\partial \Omega,\smallskip\\
\displaystyle	\int_{\Omega} u^{2} =1
	    \end{array}
    \right.
\end{equation}
in a fixed domain $\Omega$. The sequence $\{u_{k}\}$
consists of critical points at minimax levels 
(over the class of sets having arbitrary large {\sl Krasnoselskii genus})
of the energy functional
restricted to the unit sphere in $L^{2}(\Omega)$,
on which it is bounded from below.
Actually, as it is observed in \cite[Appendix]{PS}, the same result holds by adding a
nonlinearity $-|u|^{p-2}u$ in the left hand side of the equation in \eqref{eq:equacionBF}, for $p\in (2,10/3)$.

Even though Benci and Fortunato work on the unit sphere in $L^{2}(\Omega)$, their result
is true for every value of radius $\rho>0$, namely on $\mathcal M_{\rho}(\Omega)$.
 Summing up,  the result of \cite{BF} follows for the problem 
 \begin{equation*}
 \left\{
	    \begin{array}{ll}
	-\Delta u+ \phi_{u} u -|u|^{p-2}u =  \omega u & \quad \text{in }\Omega, \smallskip\\
	u =0 &\quad \text{on }\partial \Omega, \smallskip\\
\displaystyle\int_{\Omega} u^{2} =\rho^2
	    \end{array}
    \right.
\end{equation*}
for any $p\in(2,10/3), \rho>0$, and they also obtain that 
 \begin{itemize}
 \item[(a)] the solutions $u_{k}$  are possibly sign changing, except of course the solution which minimizes the related functional $I(\, \cdot \, ;\Omega)$ restricted to $\mathcal M_{\rho}(\Omega)$,
 \item[(b)] the sequence of critical levels $\{I(u_{k};\Omega)\}$ is diverging,
  \item[(c)] the sequence of Lagrange multipliers $\{\omega_{k}\}$ is diverging.
 \end{itemize}

\medskip

Our paper gives a contribution in the sense that  solutions  different from those found in \cite{BF}
are furnished here. Indeed for large (expanding) domains we find solutions  that
\begin{itemize}
\item[(a')] are all positive,
\item[(b')]  are  at a negative level of the energy functional,
\item[(c')]  all have  negative related   Lagrange multipliers,
\item[(d')] when the domain is larger and larger the Lagrange multipliers, as well the solutions $u$, are  converging
to the solutions in the whole space.
\end{itemize}

However we can prove this just for solutions with small $L^{2}$-norm  and our result just follows for $p\in(2,3)$
and not in the whole range $(2,10/3)$.

\medskip

Our approach is slightly different  from that in \cite{BF} in the sense that we use another topological invariant which also adapts to non even functionals and permits to gain the positivity of the solutions. Indeed, by means of  the Ljusternick-Schnirelmann category
and the barycenter map introduced by Benci and Cerami in \cite{BC}, we are able to find a number of solutions depending on the ``topological complexity'' of the domain $\lambda\Omega$. The main difficulty here is due to the presence  of a nonlocal term driven by the Green function;
in fact, fine estimates have to be found in order to deal with  the Green function and  compare suitable infima
when the domain  is  expanding, and implement the barycenter map on suitable sublevels of the energy functional.

As it will be clear in Section \ref{sec:prelim}, the restrictions on $p,\rho$
and $\lambda$ in the main theorem are due to:  {\em(i)} the boundedness from below of the functional, 
{\em (ii)} the evaluation of the barycenters of functions in suitable sublevels of the functional $I(\, \cdot\,; \lambda \Omega)$, and {\em (iii)} the radial property  of the ground state solution 
$\mathfrak w_{\infty}$ for the limit problem in the whole $\mathbb R^{3}$.

We observe that  problems in expanding domains have attracted attention in the ma\-the\-ma\-ti\-cal literature and also have been studied in other contexts; see e.g., \cite{A,FPS}.
However this is the first paper where, besides a multiplicity result, a convergence result is also presented.

\medskip

\subsection{Organization of the paper}
The organization of the paper is the following: in Section \ref{sec:prelim} we introduce basic notations, we recall some well known facts and give the variational framework
of the problem. In Section \ref{sec:bary} the barycenter map is introduced; it will be a fundamental tool 
in order to employ the Ljusternick-Schnirelmann theory. Here is seen the role played by large values of $\lambda$.
 In Section \ref{sec:finale}
the proof of Theorem \ref{th:main} is given.
In a final Appendix we show the striking difference between the whole space and  a bounded domain 
for what concerns the minimum of the functional
restricted to the $L^{2}$-sphere.

\section{Some preliminaries}\label{sec:prelim}
Let us start by introducing few notations. 

In all the paper $\Omega \subset \mathbb{R}^3$ is a (smooth and) bounded domain such that $0\in \Omega$. Hereafter the open ball centered in $x_{0}\in \mathbb R^{3}$ with radius $\widehat{r}>0$ will be denoted with $B_{\widehat{r}}(x_{0})$. Note that, for $\lambda >1$ and $x_0\neq 0$, we have $\lambda B_{\widehat{r}}(x_{0})=B_{\lambda \widehat{r}}(\lambda x_{0})$, so that $\lambda B_{\widehat{r}}(x_{0})\neq B_{\lambda \widehat{r}}(x_{0})$. We use $d(x,D)$ to denote the distance between a point $x\in \mathbb R^{3}$ and $D\subset \mathbb R^{3}$. Also let us fix, from now on, a small $r>0$ such that $B_{r}:=B_{r}(0)\subset \Omega$ and, setting 

\begin{align*}
    \Omega ^+_r & := \Big\{ x\in \mathbb R^3: d(x,\Omega)\leq r \Big\} \\
    \Omega ^-_r & := \Big\{ x \in \Omega: d(x,\partial \Omega ) \geq r \Big\}\,,
\end{align*}
the sets
\begin{equation}\label{II_4.9}
    \Omega,\, \Omega ^+_r, \text{ and }  \Omega ^-_r \text{ are homotopically equivalent}.
\end{equation} 
In particular, for any $\lambda>1$, the same is true for
\[
\lambda\Omega,\ \ \   \lambda\Omega_{r}^{+}
\quad \text{ and } \ \  \ \lambda\Omega_{r}^{-}.
\]
In all the paper $r$ will be the fixed value above.

 The symbol $o_s(1)$ denotes a quantity which goes to zero as $s\to +\infty$.
 
We denote with $|\cdot |_{L^{q}(D)}$ the usual $L^{q}$-norm,
and 
let $H^1_{0} (D)$ be the usual Sobolev space endowed with  equivalent (squared) norm 
\[
 \norm{u}^2_{H^1_0(D)}= \displaystyle\int_D \abs{\nabla u}^2 \text{d} x.
\]

As noted earlier, for every fixed $u\in H^{1}_{0}(D)$, the problem
\begin{equation}\label{eq:P}
    \left\{
        \begin{array}{ll}
             -\Delta \phi=u^2 & \text{in }D,\smallskip\\
             \phi=0 & \text{on }\partial D
        \end{array}
    \right.
\end{equation}
possesses a unique solution $\phi _u \in H^1_0(D)$. 
A list of properties of  $\phi_{u}$ can be found e.g., in \cite[Lemma 1.1]{Alves}.
We just observe here that, if $\mu_{1}(D)$ denotes the first eigenvalue of the Laplacian in the domain $D$,
\[
    \int_D \abs{\nabla \phi _u}^2 \text{d} x= \int_D \phi_u u^2 \text{d} x\leq \abs{\phi_u}_{L^2(D)}\abs{u^2}_{L^2(D)}\leq \mu_{1}^{-1/2}(D)\norm{\phi_u}_{H^1_0(D)}\norm{u}^2_{H^1_0(D)}
\]
and thus,
\begin{equation} \label{eq:convergencia}
    \int_D \phi_u u^2 \text{d} x\leq \mu_{1}^{-1}(D)\norm{u}^4_{H^1_0(D)}.
\end{equation}

We use the convention that given a function in $H^{1}_{0}(D)$ we will denote with the same letter the function trivially extended to the whole $\mathbb R^{3}$, which then belongs to $H^{1}(\mathbb R^{3})$.
In particular, for $u\in H^{1}_{0}(D)\subset H^{1}(\mathbb R^{3})$ if $\phi_{u}\in H^{1}_{0}(D)$ is the unique solution of \eqref{eq:P},
we also have $\phi_{u}\in H^{1}(\mathbb R^{3})$ and
\[
     -\Delta \phi_{u}=u^2 \text{ in }D,\quad \phi_{u}=0 \text{ in }\mathbb R^{3}\setminus D.
\]
However, it is important to establish explicitly the difference between $\phi_u$ and the unique solution $\varphi_u$ of the problem
\[
    -\Delta \varphi=u^2 \text{ in }\mathbb R^{3}.
\]
We will formulate the relationship between these two functions (see Fact 4 below).

\medskip

We know that the solutions $(u,\omega)$ of the problem
\[
    \left\{
	    \begin{array}{ll}
	-\Delta u+ \phi_{u} u-\abs{u}^{p-2}u =  \omega u & \quad \text{in }D,\smallskip\\
	u>0 &\quad \text{in }D,\smallskip\\
	u =0 &\quad \text{on }\partial D,\smallskip\\
		   \displaystyle     \int_{D} u^{2} \text{d} x =\rho^2,
	    \end{array}
    \right.
\]
in a smooth bounded domain $D$ can be characterized as critical points of the 
$C^{1}$ functional in $H^{1}_{0}(D)$
\[
    I(u; D):= \frac{1}{2}\int_D \abs{\nabla u}^2 \text{d} x+\frac{1}{4}\int_D \phi_u u^2 \text{d} x-\frac{1}{p}\int_D \abs{u}^p \text{d} x
\]
restricted to the manifold
\[
   \mathcal M_ \rho (D):= \Big\{ u\in H^1_0(D): \int_D u^2 \text{d} x=\rho^2\Big\}, \qquad \rho>0.
\]
Analogously, in case of the whole space, the solutions $(u,\omega)$ of the problem
\begin{equation*}
 \left\{
	    \begin{array}{ll}
	       -\Delta u+ \varphi_{u} u-\abs{u}^{p-2}u =  \omega u & \quad \text{in }\mathbb R^3, \smallskip\\
	       u>0 &\quad \text{in }\mathbb R^3,\smallskip\\
	       \displaystyle\int_{\mathbb R^3} u^{2} \text{d} x =\rho^2
	    \end{array}
    \right.
\end{equation*}
can be characterized as critical points of the  $C^{1}$ functional in $H^{1}(\mathbb R^3)$
\[
    I(u; \mathbb R^3):= \frac{1}{2}\int_{\mathbb R^3} \abs{\nabla u}^2 \text{d} x+\frac{1}{4}\int_{\mathbb R^3} \varphi_u u^2 \text{d} x-\frac{1}{p}\int_{\mathbb R^3} \abs{u}^p \text{d} x
\]
restricted to the manifold
\[
   \mathcal M_ \rho (\mathbb R^3):= \Big\{ u\in H^1(\mathbb R^3): \int_{\mathbb R^3} u^2 \text{d} x=\rho^2\Big\}, \qquad \rho>0.
\]
In both cases, the Lagrange multiplier associated to the critical point is the parameter  $\omega$
which appears in the related problem.

\medskip

We recall now the following facts that will be fundamental in the whole paper. We limit ourselves to the case $p\in(2,3)$ which is our interest here, even though some facts
 are also true for $p\in (2,10/3)$.

\medskip

{\bf Fact 1.}
For every $\rho >0$, the functional $I(\, \cdot \, ; D)$ is bounded from below and coercive on $\mathcal M_ \rho (D)$. 
This is proved in \cite[Lemma 3.1]{BS1} in the case $D=\mathbb R^{3}$, but
 it is easy to see that it holds on every domain $D$ (see, for instance, \cite[Proposition 3.3]{TMNA}).
 
 \medskip

{\bf Fact 2.}
For every $\rho>0$, it is 
\[
    c_\infty:= \inf_{u\in \mathcal M_ \rho (\mathbb R^{3})}I(u; \mathbb R^{3})\in (-\infty,0).
\]
For a proof see \cite[pp. 2498-2499]{BS2}.
The inequality $c_{\infty}<0$ is strongly based on the fact that the domain is the whole  $\mathbb R^{3}$, since 
suitable scalings are used that are not allowed in a bounded domain. In fact, in a bounded domain the infimum is strictly positive (at least for small $\rho$) as we will show in the Appendix.

In \cite{BS1} it has been proved that there exists $\rho_1=\rho_1(p) >0$ such that for any $\rho\in(0,\rho_1)$, all the minimizing sequences for $c_\infty$ are precompact in $H^1(\mathbb R^3)$ up to translations, and converging to a positive ground state
$\mathfrak w_{\infty}$ in $H^{1}(\mathbb R^{3})$,
which is the one  appearing in  the statement  of Theorem \ref{th:main}, so that
\[
 c_\infty= \min_{u\in \mathcal M_ \rho (\mathbb R^{3})}I(u; \mathbb R^{3})=I(\mathfrak w_{\infty};\mathbb R^{3})\in (-\infty, 0).
\]

More explicitly, if $\{ u_n\} \subset \mathcal M_ \rho (\mathbb R^3)$ is a minimizing sequence for $c_\infty$ and vanishing occurs, then  $u_n \rightharpoonup 0$ and $\|u_{n}\| \nrightarrow 0$ in $H^1 (\mathbb R^3)$, since $c_\infty<0$. Then the well known Lions' lemma \cite{L}   implies that  
   \[
    \sup_{y\in \mathbb R^3} \int_{B_1(y)} u_n^2 \text{d} x \geq \delta >0, \quad \text{for some } \delta >0.
\]
Hence there exist $\{ y_n\} \subset \mathbb R^3$  and $\mathfrak w_\infty \in \mathcal M_\rho (\mathbb R^3)$ such that
    $v_n:=u_n(\cdot +y_n) \in\mathcal M_\rho (\mathbb R^3)$ and  
    \begin{equation*}
    v_n \to \mathfrak w_\infty \quad  \text{ in }H^1(\mathbb R^3)\,, \qquad
    I(\mathfrak w_{\infty} ; \mathbb R^3)=c_\infty.
    \end{equation*}
    Furthermore, $|y_{n}| \to +\infty$. Indeed, since
\[
\int_{B_1} u_n^2(x+y_n) \text{d} x=\int_{B_1(y_n)} u_n^2 \text{d} x \geq \sup_{y\in \mathbb R^3} \int_{B_1(y)} u_n^2 \text{d} x - o_n(1) \geq \frac{\delta}{2} >0,
\]
if the sequence $\{ y_n\}$ were bounded, for a large $ R>0$,
\[
    \int_{B_{{R}}} u_n^2 \text{d} x \geq \int_{B_1(y_n)} u_n^2 \text{d} x \geq \frac{\delta}{2} >0, \quad \text{for every } n\in \mathbb N
\]
which is a contradiction since $u_{n}\to 0$ in $L^{2}(B_{ R})$.

Finally, as proved in  \cite[Theorem 0.1]{GPV}, the ground state $\mathfrak w_{\infty}$ 
is radially symmetric for $\rho \in (0,\rho_1)$, if necessary by reducing $\rho_1$.

\medskip

{\bf Fact 3.} There is a Lagrange multiplier $\omega_\infty$ associated to $\mathfrak w_\infty$, that satisfies
    \begin{equation}\label{lagrange:infinito}
        \rho^2 \omega_\infty=\int_{\mathbb R^3} \abs{\nabla \mathfrak w_\infty}^2 \text{d} x +\int_{\mathbb R^3}\varphi _{\infty} \mathfrak w_\infty^2 \, \text{d} x- \int_{\mathbb R^3} \mathfrak w_\infty^p \text{d} x,
    \end{equation}
    where we have written for simplicity $\varphi_\infty:=\varphi_{\mathfrak w_\infty}$.
    Let us recall the argument which shows that $\omega_{\infty}$ is negative.
   By \cite{GPV} we know that, setting
    \[
        \mathfrak w_\infty=\rho^{\frac{4}{4-3(p-2)}}\mathfrak v(\rho^{\frac{2(p-2)}{4-3(p-2)}}\ \cdot \ ),
    \]
    $\mathfrak v$ is a radial constrained (to the $L^2$-sphere in $H^1(\mathbb R^3)$) minimizer for the functional defined on $H^1(\mathbb R^3)$ by
    \[
    J(u):=\frac{1}{2}\int_{\mathbb R^3}\abs{\nabla u}^2\text{d} x+\frac{\rho^{\alpha(p)}}{4}\int_{\mathbb R^3}\varphi_uu^2\text{d} x-\frac{1}{p}\int_{\mathbb R^3}\abs{u}^p\text{d} x
    \]
    where
    $$ 
   \quad \alpha(p):=\frac{8(3-p)}{10-3p}>0\quad  \text{ if }\quad 2<p<3.
    $$
Then
    \[
        J(\mathfrak v)=\min\big\{J(u):u\in H^1(\mathbb R^3) \text{ and } \abs{u}_{L^2(\mathbb R^3)}=1\big\}
    \]
and $\mathfrak v$ is a solution of the problem
    \[
        \left\{
	       \begin{array}{ll}
	    	  -\Delta v +\rho^{\alpha(p)}\varphi_v v-\abs{v}^{p-2}v =  \omega v & \quad \text{in }\mathbb R^3,\smallskip\\
	           v>0 &\quad \text{in }\mathbb R^3,\smallskip\\
	            \displaystyle \int_{\mathbb R^{3}} v^{2}=1, &
	       \end{array}
    \right.
    \]
with $\omega$ as its  Lagrange multiplier, which evidently  satisfies
    \begin{equation}\label{lagrange:v}
        \omega=\int_{\mathbb R^3} \abs{\nabla \mathfrak v}^2 \text{d} x +\rho^{\alpha(p)}\int_{\mathbb R^3}\varphi _{\mathfrak v} \mathfrak v^2 \, \text{d} x- \int_{\mathbb R^3} \mathfrak v^p \text{d} x.
    \end{equation}
Possibly by reducing  $\rho_{1}$, for $\rho\in(0,\rho_{1})$  it is  $\omega<0$; see \cite[Proposition 1.3]{GPV}.
Note that
\[
\int_{\mathbb R^3} \abs{\nabla \mathfrak w_\infty}^2 \text{d} x=\rho^{\gamma(p)}\int_{\mathbb R^3} \abs{\nabla \mathfrak v}^2 \text{d} x\quad \text{and}\quad \int_{\mathbb R^3} \mathfrak w_\infty^p \text{d} x=\rho^{\gamma(p)}\int_{\mathbb R^3} \mathfrak v^p \text{d} x,\quad
\]
where
\[
     \gamma(p):=\frac{8-2(p-2)}{10-3p}>0\quad \text{if}\quad 2<p<3,
\]
and
 \[
 \int_{\mathbb R^3}\varphi _{\infty} \mathfrak w_\infty^2 \, \text{d} x =\rho^{\alpha(p)+\gamma(p)}\int_{\mathbb R^3}\varphi _{\mathfrak v} \mathfrak v^2 \, \text{d} x.
\]
Thus, by using \eqref{lagrange:infinito} and \eqref{lagrange:v},  we infer that
    \begin{align*}
        \rho^2 \omega_\infty&=\rho^{\gamma(p)}\bigg[\int_{\mathbb R^3} \abs{\nabla \mathfrak v}^2 \text{d} x +\rho^{\alpha(p)}\int_{\mathbb R^3}\varphi _{\mathfrak v} \mathfrak v^2 \, \text{d} x- \int_{\mathbb R^3} \mathfrak v^p \text{d} x\bigg]\\
        &=\rho^{\gamma(p)}\omega.
    \end{align*}
    Therefore, for $\rho\in(0,\rho_{1})$, it is also $\omega_\infty<0.$

    \medskip

{\bf Fact 4.}  Let us consider a smooth bounded domain $D\subset \mathbb R^3$ and $u\in H^1_0(D)$. The unique solution $\phi_u$ to the problem
    \[
        -\Delta \phi= u^2 \text{ in } D,\quad \phi=0 \text{ on } \partial D
    \]
    and the unique solution $\varphi_u$ to the problem
    \[
        -\Delta \varphi= u^2 \text{ in } \mathbb R^3
    \]
    are such that
    \begin{equation}\label{eq:relacion}
        \int_{D}\phi_u u^2\text{d} x=\int_{D}\varphi_u u^2\text{d} x-\int_{D\times D}H(x,y;D)u^2(x)u^2(y)\text{d} x\text{d} y,
    \end{equation}
    where $H(\, \cdot \, ,\, \cdot \,;D)$ denotes the \emph{smooth part} of the Green's function for $-\Delta$ in $D$ (see \cite{Iorio}). This function has the following relevant properties useful for our purpose.
    \begin{enumerate}
    \item The function $H(\, \cdot \, ,\, \cdot \,;D)$ is nonnegative.
     \item We have
        \begin{equation}\label{eq:M}
            M_D:=\abs{H(\, \cdot \, ,\, \cdot \,;D)}_{L^\infty(D\times D)}<+\infty.
        \end{equation}
    \item \label{partesuave} If $\lambda>1$, then
        \[
            H(x,y;\lambda D)=\frac{1}{\lambda}H\bigg(\frac{x}{\lambda},\frac{y}{\lambda};D\bigg),\quad \forall \ (x,y)\in (\lambda D)\times (\lambda \overline D).
        \]
    \item If $D_1\subset D_2\subset \mathbb R^3$ are smooth bounded domains, then
        \begin{equation}\label{partesuave2}
            H(x,y;D_2)<H(x,y;D_1),\quad \forall \ (x,y)\in D_1\times \overline D_1.
        \end{equation}
    \end{enumerate}
    As two important consequences, we highlight the following:
    \begin{enumerate}
    \item[i)] Given a family $\{u_\lambda\}_{\lambda > 1}$ of functions such that $u_\lambda \in \mathcal M_\rho(\lambda D)$, using the property (\ref{partesuave}) of the function $H$, we have  the estimate
    \begin{equation}\label{estimativa}
        0\leq \int_{\lambda D\times \lambda D} H(x,y; \lambda D)u_\lambda^2(x)u_\lambda^2(y)\text{d} x\text{d} y\leq \\ \frac{M_D}{\lambda} \bigg(\int_{\lambda D}u_\lambda^2(x)\text{d} x\bigg)^2 =\frac{M_D}{\lambda} \rho^4.
    \end{equation}
    \item[ii)] If $D_1\subset D_2\subset \mathbb R^3$ are smooth bounded domains, using the trivial extension of functions, if $u\in H^1_0(D_1)$, then $u\in H^1_0(D_2)$, and
    \[
        \int_{D_1}\phi_1 u^2< \int_{D_2}\phi_2 u^2,
    \]
    where $\phi_i$ for $i\in \{1,2\}$ denotes the unique solution of the problem
    \[
        -\Delta \phi= u^2 \text{ in } D_i,\quad \phi=0 \text{ on } \partial D_i.
    \]
    Therefore
    \(
        I(u;D_1)<I(u;D_2).
    \)
    \end{enumerate}

\medskip

\section{The role of large \texorpdfstring{$\lambda$}{lambda}}\label{sec:bary}

In this section we consider some interesting and fundamental results which are true
specifically for expanding domains.
All the results involve a limit in $\lambda$ and the role of taking a large $\lambda$ is then evident.

\medskip

Let us start with few notations.
For $u\in H^1(\mathbb R^3)$ with compact support we define the barycenter map by
\[
    \beta (u):=\frac{\displaystyle\int_{\mathbb R^3} x\abs{\nabla u}^2\text{d} x}{\displaystyle\int_{\mathbb R^3} \abs{\nabla u}^2\text{d} x}\in \mathbb R^{3}.
\]
It is a continuous map.\\
For $x\in \mathbb R^3$ and  $0<\widehat r<\widehat R$ , $A_{\widehat R,\widehat r,x} := B_{\widehat R}(x)\setminus \overline{B_{\widehat r}(x)}$
is the annulus
centered in $x$ and radii $\widehat r,\widehat R$; for $\lambda>1$ and $x\neq 0$, it holds that $\lambda A_{\widehat R,\widehat r,x}= A_{\lambda \widehat R,\lambda \widehat r,\lambda x}$, from which
$\lambda A_{\widehat R,\widehat r,x}\neq A_{\lambda \widehat R,\lambda \widehat r,x}$.
If $x=0$, we simply write $A_{\widehat R,\widehat r}$.
 For $\lambda >1$ and $R\in (r,+\infty)$, where $r$ is the fixed number in the beginning of Section \ref{sec:prelim}, the set 
\[
    \Big\{  u\in \mathcal M_\rho (A_{\lambda R,\lambda r,x}):\ \beta(u)=x\Big\} 
\]
is not empty and 
\[
    a( R, r, \lambda,x):= \inf \Big\{ I(u; A_{\lambda R,\lambda r,x}): u\in \mathcal M_\rho (A_{\lambda R,\lambda r,x}),\ \beta(u)=x\Big\}
    >-\infty.
\]

Note that the above infimum does not depend on the choice of the point $x$; indeed, for any $x\in \mathbb R^3$ and $0<\widehat r<\widehat R$, we have (see Fact 4)
\[
    H(y,z;A_{R,r,x})=H(y-x,z-x;A_{R,r}),\quad \forall \ (y,z)\in A_{R,r,x}\times \overline A_{R,r,x}
\]
and then every term in the functional is invariant under translations. Let us set $a( R, r, \lambda):=a( R, r, \lambda,0)$. We also use the notations
\[
    b_\lambda:=\inf_{u\in \mathcal M_ \rho (B_{\lambda r})}I(u; B_{\lambda r})\qquad \text{and}\qquad c_\lambda:=\inf_{u\in \mathcal M_ \rho (\lambda \Omega)}I(u; \lambda \Omega).
\]
Of course, the numbers $a( R, r, \lambda)$, $b_\lambda$ and $c_\lambda$ (besides $c_\infty$) also depend on $\rho>0$. However, we do not make explicit this dependence in the notation.

\medskip

Now we use the fact that the families of infima with fixed $L^2$-norm are bounded from below. Let us choose $\lambda> 1$ and consider a smooth bounded domain $D\subset \mathbb R^3$; if $u \in \mathcal M_\rho(\lambda D)$ (so that $u \in \mathcal M_\rho(\mathbb R^3)$), then there exists $m<0$ such that
\[
    m\leq \int_{\mathbb R^3} \bigg(\frac{1}{2}\abs{\nabla u}^2-\frac{1}{p}\abs{u}^p\bigg)\text{d} x=\int_{\lambda D} \bigg(\frac{1}{2}\abs{\nabla u}^2-\frac{1}{p}\abs{u}^p\bigg)\text{d} x<I(u; \lambda D);
\]
see \cite[Lemma 3.1]{BS1}.
Therefore,
 for the different kinds of domains in which we are interested, i.e., $D=B_r$, $D=\Omega$ and $D=A_{R,r}$, it is 
$    m\leq b_\lambda, m\leq c_\lambda $ and $m\leq a(R,r,\lambda)$  for all $\lambda>1$.
As a consequence,
\[
    m\leq \liminf_{\lambda \to +\infty} b_\lambda,\quad m\leq \liminf_{\lambda \to +\infty} c_\lambda\quad \text{and}\quad m\leq \liminf_{\lambda \to +\infty}a(R,r,\lambda).
\]

\begin{remark}
    For a smooth bounded domain $D\subset \mathbb R^3$ and $u\in \mathcal M_\rho (D)$, we have by \eqref{eq:relacion}
        \begin{align*}
            I(u;D)&=\int_{D} \bigg(\frac{1}{2}\abs{\nabla u}^2-\frac{1}{p}\abs{u}^p\bigg) \emph{d} x+\frac{1}{4}\int_{D} \phi_u u^2 \emph{d} x\\
            &= \int_{D}\bigg(\frac{1}{2}\abs{\nabla u}^2+\frac{1}{4}\varphi_u u^2-\frac{1}{p}\abs{u}^p\bigg) \emph{d} x-\frac{1}{4}\int_{D\times D} H(x,y; D)u^2(x)u^2(y)\emph{d} x\emph{d} y,
        \end{align*}
    from which
    \begin{equation}\label{relacion}
    I(u;D)=I(u;\mathbb R^3)-\frac{1}{4}\int_{D\times D} H(x,y; D)u^2(x)u^2(y)\emph{d} x\emph{d} y.
    \end{equation}
\end{remark}

    Using the previous remark, we obtain the following inequalities involving the infima:
    \begin{equation}\label{inferiores}
        c_\infty \leq \liminf_{\lambda \to +\infty} a(R,r,\lambda),\quad c_\infty \leq \liminf_{\lambda \to +\infty} b_\lambda\quad \text{and} \quad c_\infty \leq \liminf_{\lambda \to +\infty} c_\lambda.
    \end{equation}
    Indeed, by definition there are  families $\{u_\lambda\}_{\lambda>1}$, $\{v_\lambda\}_{\lambda>1}$ and $\{w_\lambda\}_{\lambda>1}$  such that, for every $\lambda > 1$, $u_\lambda \in \mathcal M_\rho(A_{\lambda R,\lambda r})$,  $v_\lambda \in \mathcal M_\rho(B_{\lambda r})$, $w_\lambda \in \mathcal M_\rho(\lambda \Omega)$ (hence $u_\lambda, v_\lambda, w_\lambda \in \mathcal M_\rho(\mathbb R^3)$) and for which the inequalities
    \[
        I(u_\lambda;A_{\lambda R,\lambda r})<a(R,r,\lambda)+\frac{1}{\lambda}, \quad I(v_\lambda;B_{\lambda r})<b_\lambda+\frac{1}{\lambda},\quad I(w_\lambda;\lambda \Omega)<c_\lambda+\frac{1}{\lambda}
    \]
    hold. But then, using \eqref{relacion} with $D=A_{\lambda R, \lambda r}$, $D=B_{\lambda r}$ and $D=\lambda \Omega$, respectively,
    
    \begin{align*}
        c_\infty &\leq I(u_\lambda;\mathbb R^3)< a(R,r,\lambda)+\frac{1}{\lambda}+\frac{1}{4}\int_{A_{\lambda R,\lambda r}\, \times \, A_{\lambda R,\lambda r}} H(x,y; A_{\lambda R,\lambda r})u_\lambda^2(x)u_\lambda^2(y)\emph{d} x\emph{d} y,\\
        c_\infty &\leq I(v_\lambda;\mathbb R^3)< b_\lambda+\frac{1}{\lambda}+\frac{1}{4}\int_{B_{\lambda r}\times B_{\lambda r}} H(x,y; B_{\lambda r})v_\lambda^2(x)v_\lambda^2(y)\emph{d} x\emph{d} y,\\
        c_\infty &\leq I(w_\lambda;\mathbb R^3)< c_\lambda+\frac{1}{\lambda}+\frac{1}{4}\int_{\lambda \Omega\times \lambda \Omega} H(x,y; \lambda \Omega)w_\lambda^2(x)w_\lambda^2(y)\emph{d} x\emph{d} y,
    \end{align*}
and the conclusion easily follows by using \eqref{estimativa}.

The next two propositions will serve as a preparation for the main result of this section, Proposition \ref{propII_4.1}. Recall that $\rho_1>0$ was previously introduced in Section \ref{sec:prelim}.
\begin{prop}\label{propII_2.1}Let $2<p<3$. For any  $\rho\in(0,\rho_1)$ it is
    \[
        \liminf_{\lambda \to +\infty} a( R, r, \lambda) > c_\infty.
    \]
\end{prop}

\begin{proof}
Assume by contradiction that, along a subsequence $\lambda_n \to +\infty$ it holds
    \[
        \lim_{\lambda \to +\infty} a(R,r,\lambda_n)=c_\infty.
    \]
Besides, consider a sequence $\{u_n\}$ with $u_n\in \mathcal M_{\rho}(A_{\lambda_n R,\lambda_n r})$, such that $\beta(u_n)=0$ and it also satisfies
    \[
        a(R,r,\lambda_n)\leq I(u_n;A_{\lambda_n R,\lambda_n r}) <a(R,r,\lambda_n) + o_n(1).
    \]
    Hence, using \eqref{relacion} with $D=A_{\lambda_n R,\lambda_n r}$, we have
    \[
        a(R,r,\lambda_n)+o_n(1)\leq I(u_n;\mathbb R^3) <a(R,r,\lambda_n) + o_n(1)
    \]
    so that, by our assumption,
    \[
        I(u_n;\mathbb R^3) \to c_\infty.
    \]
Thus, we may suppose that 
    \begin{equation}\label{eq:conv}
    \{ u_n\} \subset \mathcal M_\rho (\mathbb R^3),\quad \beta(u_n)=0 \ \text{ and } \ I(u_n; \mathbb R^3) 
    \to c_\infty.
    \end{equation}
Due to the coercivity of $I(\, \cdot \, , \mathbb R^{3})$,  the sequence $\{u_n\}$
has to be bounded. Moreover, $\norm{u_n} \nrightarrow 0$; otherwise, by using \eqref{eq:convergencia},
    \[
        \int_{\mathbb R^3} \abs{\nabla u_n}^2 \text{d} x,\ \int_{\mathbb R^3} \varphi_{u_n}u_n^2 \text{d} x,\ \int_{\mathbb R^3} \abs{u_n}^p 
        \text{d} x \to 0
    \]
 and    by \eqref{eq:conv} 
    \[
c_\infty +      o_{n}(1)=I(u_n; \mathbb R^3)= \frac{1}{2}\int_{\mathbb R^3} \abs{\nabla u_n}^2 \text{d} x+\frac{1}{4}\int_{\mathbb R^3} \varphi_{u_n}u_n^2 \text{d} x-\frac{1}{p}\int_{\mathbb R^3} \abs{u_n}^p \text{d} x= o_n(1),
    \]
we obtain $c_\infty=0$, in contradiction with $c_{\infty}<0$. Hence, as in Fact 2 of Section \ref{sec:prelim}, 
 there exists a sequence $\{ y_n\} \subset \mathbb R^3$ with $|y_{n}|\to +\infty$ such that $v_n:=u_n(\cdot +y_n) \in \mathcal M_\rho (\mathbb R^3)$ and
    \begin{equation}\label{eq:convergence2}
    v_n \to \mathfrak w_\infty \quad  \text{ in }H^1(\mathbb R^3)\,, \qquad
    I(\mathfrak w_{\infty} ; \mathbb R^3)=c_\infty.
    \end{equation}
The convergence in \eqref{eq:convergence2} can equivalently be written as
$w_n:= u_n(\cdot +y_n)-\mathfrak w_\infty\to 0$; i.e.,
    \[
        u_n(x)=w_n(x-y_n)+\mathfrak w_\infty (x-y_n)\quad \text{ with } \ w_{n}\to 0 \ \text{ in } H^{1}(\mathbb R^{3}).
    \]
Since $I(\, \cdot \, ;A_{\lambda_n R,\lambda_n r})$ is rotationally invariant, we can assume that
    \[
        y_n=(y_n^1,0, 0)\quad \text{with} \quad y_n^1<0.
    \]

\medskip

{\bf Claim 1: } $ \displaystyle\int_{B_{\lambda_n r/2}(y_n)}\abs{\nabla u_n}^2\text{d} x \to \int_{\mathbb R^3}\abs{\nabla \mathfrak w_\infty}^2 \text{d} x$.

\medskip

 In fact, denoting with
    \begin{align*}
        \sigma_{1,n}&:=\int_{B_{\lambda_n r/2}(y_n)}\abs{\nabla u_n}^2\text{d} x =\int_{B_{\lambda_n r/2}(y_n)}\abs{\nabla(w_n(x-y_n)+\mathfrak w_\infty (x-y_n))}^2 \text{d} x\,,\\
        \sigma_{2,n}&:=\int_{B_{\lambda_n r/2}(y_n)} 2\abs{\nabla w_n(x-y_n)} \abs{\nabla \mathfrak w_\infty(x-y_n)}\text{d} x\,,\\
        \sigma_{3,n}&:=\int_{B_{\lambda_n r/2}(y_n)}\abs{\nabla w_n(x-y_n)}^2\text{d} x\,, \\
        \sigma_{4,n}&:=\int_{B_{\lambda_n r/2}(y_n)}\abs{\nabla \mathfrak w_\infty(x-y_n)}^2\text{d} x\,,
    \end{align*}
    we have
\begin{equation}\label{eq:abcd}
        \sigma_{3,n}-\sigma_{2,n}+\sigma_{4,n} \leq \sigma_{1,n} \leq \sigma_{3,n}+\sigma_{2,n}+\sigma_{4,n}.
\end{equation}

With the change of variable  $z:=x-y_n$ we easily deduce that
    \begin{align*}
        \sigma_{2,n}&=\int_{B_{\lambda_n r/2}} 2\abs{\nabla w_n}\abs{\nabla \mathfrak w_\infty}\text{d} z \leq 2\norm{w_n}_{H^{1}
        (\mathbb R^{3})}\Big( \int _{\mathbb R^3}\abs{\nabla \mathfrak w_\infty}^2\text{d} x\Big) ^{1/2}\to 0,\\
        \sigma_{3,n}&=\int_{B_{\lambda_n r/2}} \abs{\nabla w_n}^2\text{d} z \leq \norm{w_n}_{H^{1}(\mathbb R^{3})}^2\to 0,\\
        \sigma_{4,n}&=\int_{B_{\lambda_n r/2}} \abs{\nabla \mathfrak w_\infty}^2\text{d} z\to \int _{\mathbb R^3}\abs{\nabla \mathfrak w_\infty}^2\text{d} x.
    \end{align*}
    Thus, by \eqref{eq:abcd}, we get the claim. On the other hand, setting  $\Theta _n:=  A_{\lambda_n R,\lambda_n r} \cap B_{\lambda_n r/2}(y_n)$, we have
    \[
        \int_{B_{\lambda_n r/2}(y_n)} \abs{\nabla u_n}^2\text{d} x=\int_{\Theta _n} \abs{\nabla u_n}^2\text{d} x+\int_{B_{\lambda_n r/2}(y_n)\setminus \Theta _n} \abs{\nabla u_n}^2\text{d} x.
    \]
    But $B_{\lambda_n r/2}(y_n)\setminus \Theta _n=B_{\lambda_n r/2}(y_n)\setminus A_{\lambda_n R,\lambda_n r}$ and, being
$\textrm{supp}\ u_n\subset A_{\lambda_n R,\lambda_n r}$, from Claim 1
\[
    \displaystyle\int_{B_{\lambda_n r/2}(y_n)\setminus \Theta _n} \abs{\nabla u_n}^2\text{d} x=0
\]
holds. As a consequence, $\displaystyle\int_{\Theta _n} \abs{\nabla u_n}^2\text{d} x \to \int _{\mathbb R^3}\abs{\nabla \mathfrak w_\infty}^2\text{d} x$.
In virtue of this we have

{\bf Claim 2: } $   \displaystyle   \int_{\Upsilon _n} \abs{\nabla u_n}^2\text{d} x \to 0$, where $\Upsilon _n := A_{\lambda_n R,\lambda_n r} \setminus B_{\lambda_n r/2}(y_n).$

    \medskip
    
    Indeed, the inequality
    \begin{multline*}
        \int_{A_{\lambda_n R,\lambda_n r}} \abs{\nabla u_n}^2\text{d} x \leq \int_{A_{\lambda_n R,\lambda_n r}}\abs{\nabla w_n(x-y_n)}^2\text{d} x+\int_{A_{\lambda_n R,\lambda_n r}} 2\abs{\nabla w_n(x-y_n)} \abs{\nabla \mathfrak w_\infty(x-y_n)}\text{d} x \\+\int_{A_{\lambda_n R,\lambda_n r}}\abs{\nabla \mathfrak w_\infty (x-y_n)}^2 \text{d} x
    \end{multline*}
    can be rewritten as
    \begin{multline*}
    \int_{A_{\lambda_n R,\lambda_n r}} \abs{\nabla u_n}^2\text{d} x \leq \int_{\mathbb R^3}\abs{\nabla w_n(x-y_n)}^2\text{d} x+\int_{\mathbb R^3} 2\abs{\nabla w_n(x-y_n)} \abs{\nabla \mathfrak w_\infty(x-y_n)}\text{d} x \\+\int_{\mathbb R^3}\abs{\nabla \mathfrak w_\infty (x-y_n)}^2 \text{d} x.
    \end{multline*}
    Hence, due to the invariance under translations of the integrals,
    \begin{align*}
        \int_{A_{\lambda_n R,\lambda_n r}} \abs{\nabla u_n}^2\text{d} x &\leq \norm{w_n}_{H^1(\mathbb R^3)}^2+2\norm{w_n}_{H^{1}
        (\mathbb R^{3})}\Big( \int _{\mathbb R^3}\abs{\nabla \mathfrak w_\infty}^2\text{d} x\Big) ^{1/2}+\int_{\mathbb R^3}\abs{\nabla\mathfrak w_\infty }^2 \text{d} x\\
        &=o_{n}(1)+\int_{\mathbb R^3}\abs{\nabla\mathfrak w_\infty }^2 \text{d} x,
    \end{align*}
    from which
    \[
        \limsup_{n \to \infty}\int_{A_{\lambda_n R,\lambda_n r}} \abs{\nabla u_n}^2\text{d} x\leq \int_{\mathbb R^3}\abs{\nabla\mathfrak w_\infty }^2 \text{d} x.
    \]
    On the other hand,
    \[
        \int_{A_{\lambda_n R,\lambda_n r}} \abs{\nabla u_n}^2\text{d} x=\int_{\Upsilon_n} \abs{\nabla u_n}^2\text{d} x+\int_{\Theta_n} \abs{\nabla u_n}^2\text{d} x\geq \int_{\Theta_n} \abs{\nabla u_n}^2\text{d} x,
    \]
    so that
    \[
        \liminf_{n \to \infty}\int_{A_{\lambda_n R,\lambda_n r}} \abs{\nabla u_n}^2\text{d} x\geq \int_{\mathbb R^3}\abs{\nabla\mathfrak w_\infty }^2 \text{d} x.
    \]
    As a consequence,
    \[
        \lim_{n \to \infty}\int_{A_{\lambda_n R,\lambda_n r}} \abs{\nabla u_n}^2\text{d} x= \int_{\mathbb R^3}\abs{\nabla\mathfrak w_\infty }^2 \text{d} x
    \]
    and from
    \[
        \int_{\Upsilon_n} \abs{\nabla u_n}^2\text{d} x=\int_{A_{\lambda_n R,\lambda_n r}} \abs{\nabla u_n}^2\text{d} x-\int_{\Theta_n} \abs{\nabla u_n}^2\text{d} x
    \]
    we have
    \[
    \int_{\Upsilon _n} \abs{\nabla u_n}^2\text{d} x \to 0,
    \]
    proving the claim.

    \medskip

Finally, since
    \begin{align*}0 &= \beta(u_n)=\int_{A_{\lambda_n R,\lambda_n r}} x^1\abs{\nabla u_n}^2\text{d} x= \int_{\Theta _n} x^1\abs{\nabla u_n}^2\text{d} x+\int_{\Upsilon _n} x^1\abs{\nabla u_n}^2\text{d} x\\
        &< -\frac{\lambda _n r}{2}\bigg(\int_{\mathbb R^3}\abs{\nabla\mathfrak w_\infty }^2 \text{d} x+o_n(1)\bigg)+\lambda _nR \int_{\Upsilon _n} \abs{\nabla u_n}^2\text{d} x
    \end{align*}
we deduce that
\[
    \frac{r}{2R}\int_{\mathbb R^3}\abs{\nabla\mathfrak w_\infty }^2 \text{d} x+o_n(1)<\displaystyle\int_{\Upsilon _n} \abs{\nabla u_n}^2\text{d} x,
\]
in contradiction with Claim 2.
    The proof is thereby completed.
\end{proof}

To prove  Theorem \ref{th:main},  we will need to deal with radial functions. Since it is not clear
if $b_\lambda$ is achieved on a radial minimizer
(see e.g., \cite[Theorem 1.7]{Ruiz}) we introduce the radial setting. For any $\rho>0$, $\lambda>1$, let
\begin{eqnarray*}
\mathcal M_{\rho }^*(B_{\lambda r})&:=& \Big\{u\in \mathcal M_\rho (B_{\lambda r}) :u\text{ is radial} \Big\}\\
&=&\Big\{ u\in H^{1}_{0}(B_{\lambda r}): |u|_{L^2(B_{\lambda r})} = \rho, \ u \textrm{ is radial}\Big\}
\end{eqnarray*}
and
\begin{equation}\label{eq:brad}
b_{\lambda}^*:=\inf_{u\in \mathcal M_{\rho }^{*}(B_{\lambda r})} I(u; B_{\lambda r}).
\end{equation}
By using the arguments of \cite{BF,ORL}, it is easy to see that $b_{\lambda}^{*}$ is achieved on a function that we denoted by $\mathfrak w_{B_{\lambda r}}^{*}$.

\begin{prop}\label{propII_3.1}
Let $2<p<3$. For any  $\rho\in(0,\rho_1)$, it is
    \[
        \lim_{\lambda \to \infty} b_{\lambda}
        = \lim_{\lambda \to \infty} b_{\lambda}^{*}
        =\lim_{\lambda \to \infty} c_\lambda 
        = c_\infty.
    \]
\end{prop}

\begin{proof}
Let $\mathfrak w_\infty$ be a positive radial ground state solution of the problem in the whole space $\mathbb R^{3}$,
    and $h: [0,+\infty) \to [0,1]$ a decreasing, $C^\infty$ function such that
    \[
        h(t):=
        \begin{cases}
            1,& t \leq 1,\\
            0,& t \geq 2.
        \end{cases}
    \]
    For  $T\geq r/2$  consider the function $h_T \in C_0^\infty (\mathbb R^3)$ given by $h _T (x):=h(\abs{x}/T)$, and
    define $w_T := \mathfrak w_\infty h_T$. Note that $w_T \to \mathfrak w_\infty$ in $H^1(\mathbb R^3)$, when $T\to +\infty$. Also let $t_T>0$ be such that $t_T w_T \in \mathcal M_\rho (B_{2T})$.

    After fixing $T\geq r/2$, the number ${\widehat \lambda}={\widehat \lambda}(T):=2T/r\geq 1$ is such that for every $\lambda \geq {\widehat \lambda}$, $B_{2T} \subset B_{\lambda r}\subset \lambda \Omega$. 
    Then for every $\lambda \geq {\widehat \lambda}$, denoting $\phi_T:=\phi_{t_Tw_T}$, by using \eqref{relacion} with $D=B_{\lambda r}$ and $D=\lambda \Omega$, we have
    \begin{align*}
        b_{\lambda}\leq b_{\lambda}^{*}  &\leq I(t_T w_T; B_{\lambda r})\\ 
        &=I(t_T w_T; \mathbb R^3)-\frac{t_T^4}{4\lambda}\int_{B_{\lambda r}\times B_{\lambda r}} H\bigg(\frac{x}{\lambda},\frac{y}{\lambda}; B _{r}\bigg)w_T^2(x)w_T^2(y)\text{d} x\text{d} y
    \end{align*}
    and
    \begin{align*}
        c_{\lambda} &\leq I(t_T w_T; \lambda \Omega)\\
        &=I(t_T w_T; \mathbb R^3)-\frac{t_T^4}{4\lambda}\int_{\lambda \Omega \times \lambda \Omega} H\bigg(\frac{x}{\lambda},\frac{y}{\lambda}; \Omega \bigg)w_T^2(x)w_T^2(y)\text{d} x\text{d} y.
    \end{align*}
    Therefore, using \eqref{estimativa} with $D=B_{r}$ and $D=\Omega$ respectively,
    \begin{equation}\label{II_4.1}
        \limsup_{\lambda \to +\infty} b_{\lambda}\leq \limsup_{\lambda \to +\infty} b_{\lambda}^{*} \leq I(t_T w_T; \mathbb R^3)\ \text{ and }\ \limsup_{\lambda \to +\infty} c_{\lambda} \leq I(t_T w_T; \mathbb R^3),\quad  \forall \,  T\geq r/2.
    \end{equation}
    Since $w_T \to \mathfrak w_\infty$ in $L^2(\mathbb R^3)$, by the definition of $t_T$ we have 
    \[
        t_T^2= \frac{\rho ^2}{\displaystyle \int_{\mathbb R^3} w_T^2\ \text{d} x} \to 1.
    \]
   It follows that $t_T\to 1$. In particular, $I(t_T w_T; \mathbb R^3) \to I(\mathfrak w_\infty; \mathbb R^3)=c_\infty$ as  $T\to +\infty$, and then,
    by \eqref{II_4.1}, we infer that
    \begin{equation*}
        \limsup_{\lambda \to +\infty} b_{\lambda}\leq \limsup_{\lambda \to +\infty} b_{\lambda}^{*} \leq c_\infty \quad \text{and}\quad \limsup_{\lambda \to+ \infty} c_{\lambda} \leq c_\infty.
    \end{equation*}
We can conclude by \eqref{inferiores}.
\end{proof}

With the above propositions in hand, we can provide a proof of the main result in this section. Given a number $l \in \mathbb R$, we define the sublevel set
\[
    [I(\, \cdot \, ; \lambda \Omega)]^l:=\Big\{u\in \mathcal M_\rho (\lambda \Omega) : I(u; \lambda \Omega)\leq l \Big\}.
\]

For the next proposition, note that $M_{B_r}>M_{\Omega}$ since $B_r\subset \Omega$, in virtue of \eqref{eq:M} and \eqref{partesuave2}.

\begin{prop}\label{propII_4.1}
    Let $2<p<3$. For any  $\rho\in(0,\rho_1)$, there exists $\Lambda=\Lambda(\rho)> 1$ such that, for every $\lambda \geq \Lambda$ it holds that
    \[
        u\in [I(\cdot \, ;\lambda \Omega)]^{l(\lambda)} \  \Longrightarrow \ \beta(u) \in \lambda \Omega ^+_r,
    \]
    where
    \[
        l(\lambda):=b_\lambda^*+\frac{1}{\lambda}\bigg(1+\frac{M_{B_r}\rho^4}{4}\bigg).
    \]
\end{prop}

\begin{proof}
For any $\lambda >1$ there is a function $u_\lambda \in \mathcal M_\rho(B_{\lambda r})$ for which
    \[
        I(u_\lambda;B_{\lambda r})<b_\lambda +\frac{1}{\lambda}.
    \]
    Then, using  
    the fact that the smooth part of the Green's function on $\lambda \Omega$ is positive
    and \eqref{relacion} with $D=B_{\lambda r}$, we get
    \begin{multline*}
        I(u_\lambda;\mathbb R^3)-\frac{1}{4}\int_{\lambda \Omega \times \lambda \Omega} H(x,y; \lambda \Omega)u_\lambda^2(x)u_\lambda^2(y)\text{d} x\text{d} y<I(u_\lambda;\mathbb R^3)\\ <b_\lambda +\frac{1}{\lambda}+\frac{1}{4}\int_{B_{\lambda r}\times B_{\lambda r}} H(x,y; B_{\lambda r})u_\lambda^2(x)u_\lambda^2(y)\text{d} x\text{d} y, 
    \end{multline*}
    from which, using \eqref{relacion} with $D=\lambda \Omega$ and the estimate \eqref{estimativa} with $D=B_r$, the inequalities
    \begin{equation}\label{novacio}
        c_\lambda \leq I(u_\lambda;\lambda\Omega)<b_\lambda +\frac{1}{\lambda}+\frac{M_{B_r}\rho^4}{4\lambda}\leq l(\lambda).
    \end{equation}
    hold. Hence, the sublevel is not empty. 
    
    We argue now by contradiction. Assume that there exists a sequence of numbers $\{\lambda_n\}$ with  $\lambda _n \to +\infty$ 
    and $u_n \in \mathcal M_\rho (\Omega _{n})$ is such that $I(u_n; \Omega _{n})\leq l(\lambda_n)$ but $x_n:= \beta (u_n) \notin \lambda _n \Omega ^+_r$. From here and until the end of the proof, let 
    $$\Omega_{n}:=\lambda_n \Omega, \quad 
    R> \mbox{diam}\ \Omega, 
    \quad A_n:=A_{\lambda_n R,\lambda_n r,x_n},\quad c_n:=c_{\lambda_n},\quad  b_{n}^{*}:=b_{\lambda_n}^{*}.$$

\smallskip

    {\bf Claim: } The inclusion $\Omega _{n} \subset A_n$ holds.

    \medskip

    \noindent Of course, from the chain of implications
    \begin{align*}
        x_n \notin \lambda _n \Omega ^+_r \ & \Longrightarrow \ d(x_n / \lambda _n, \Omega) >r \ \Longrightarrow \ d(x_n, \Omega _{n}) > \lambda _n r \ \Longrightarrow \  B_{\lambda_n r} (x_n) \cap \Omega _{n} =\emptyset\\
        \ & \Longrightarrow \ \overline{ B_{\lambda_n r} (x_n)} \cap \Omega _{n} =\emptyset
    \end{align*}
    we obtain that
    \begin{equation}\label{II_4.6}
        \Omega _{n} \subset \mathbb R^3 \setminus \overline{ B_{\lambda_n r} (x_n)}.
    \end{equation} 
    Now let $y_n \in \Omega _{n}$ be an arbitrary point. Note that
    \[
        z_n\in \Omega _{n}\ \Longrightarrow \abs{z_n-y_n} < \mbox{diam}\ \Omega _{n} = \lambda_ n\, \mbox{diam}\ \Omega < \lambda _n R \ \Longrightarrow \ z_n\in  B_{\lambda_n R}(y_n).
    \]  
    Since $\textrm{supp } u_n \subset \Omega _{n} \subset  B_{\lambda_n R}(y_n)$, 
this implies that $x_n=\beta (u_n) \in  B_{\lambda_n R}(y_n)$. Therefore, $y_n \in B_{\lambda_n R}(x_n)$
and, by the arbitrariness of $y_{n}$, it follows that
\begin{equation}\label{eq:contenido}
\Omega _{n} \subset B_{\lambda _n R}(x_n).
\end{equation}
    From \eqref{II_4.6} and \eqref{eq:contenido} we get the Claim. 

\medskip

    We assert now that
    \begin{equation}\label{II_4.8}
        a(R, r,\lambda _n, x_n)< c_{n}+\frac{1}{\lambda_n}+\frac{M_{B_r}\rho^4}{4\lambda_n} <b_n^*+\frac{2}{\lambda_n}+\frac{M_{B_r}\rho^4}{2\lambda_n}.
    \end{equation}
    Indeed, we can take a sequence $\{v_n\}$ with $v_n\in \mathcal M_\rho (\Omega _{n})$ such that
    \[
        I(v_n;\Omega_n)<c_n+\frac{1}{\lambda_n}.
    \]
    Then by the Claim, the fact that the smooth part of the Green's function on $A_n$ is positive, and  \eqref{relacion} with $D=\Omega_n$, we deduce that
    \begin{multline*}
        I(v_n;\mathbb R^3)-\frac{1}{4}\int_{A_n\times A_n} H(x,y; A_n)v_n^2(x)v_n^2(y)\text{d} x\text{d} y<I(v_n;\mathbb R^3)\\ <c_n +\frac{1}{\lambda_n}+\frac{1}{4}\int_{\Omega_n \times \Omega_n} H(x,y; \Omega_n)v_n^2(x)v_n^2(y)\text{d} x\text{d} y.
    \end{multline*}
    From these inequalities, using \eqref{relacion} with $D=A_n$, the estimate \eqref{estimativa} with $D=\Omega$, and \eqref{partesuave2} with $D_1=B_r$, $D_2=\Omega$, we get
   \[
        a(R, r,\lambda _n, x_n)\leq I(v_n;A_n) <c_n +\frac{1}{\lambda_n}+\frac{M_{B_r}\rho^4}{4\lambda_n}.
   \]
   Therefore, in virtue of \eqref{novacio} we infer \eqref{II_4.8}. But then, since $a(R, r,\lambda _n, x_n)=a( R, r,\lambda _n)$,  we can write
    \[
        a( R, r,\lambda _n)<b_n^*+\frac{2}{\lambda_n}+\frac{M_{B_r}\rho^4}{2\lambda_n}
    \]
    which, together with Proposition \ref{propII_3.1}, implies that
    \[
        \liminf_{n\to \infty} a( R, r,\lambda _n) \leq \liminf_{n\to \infty} b_{n}^{*}= 
        c_ \infty.
    \]
This contradicts Proposition \ref{propII_2.1}, so that the proof is ended.
\end{proof}

We recall that $\mathfrak w_{B_{\lambda r}}^*  \in \mathcal M_{\rho}^{*}(B_{\lambda r})$  denotes a  
positive, radial ground state; i.e., $\mathfrak w_{B_{\lambda r}}^*$ is such that
\[
I(\mathfrak w_{B_{\lambda r}}^*; B_{\lambda r})=b_{\lambda}^{*};
\]
see \eqref{eq:brad}.
We now consider the continuous map $\Psi _{\lambda, r}: \lambda \Omega ^-_r \to H^1_0(\lambda \Omega)$,  given by
\[
    [\Psi _{\lambda, r}(y)](x):=\begin{cases}
        \mathfrak w_{B_{\lambda r}}^*(\abs{x-y}), & \text{if } x\in  B_{\lambda r}(y),\\
        0, & \text{if }x\in \lambda \Omega \setminus B_{\lambda r}(y),
    \end{cases}
    \quad \text{for every } y \in \lambda \Omega ^-_r.
\]

\medskip
Let us fix $y\in \lambda \Omega^{-}_r$; keeping in mind the terms of the functional, we make explicit the relation between integrals involving $\Psi _{\lambda, r}(y)$ and the corresponding integrals involving $\mathfrak w_{B_{\lambda r}}^*$. For the gradient terms we have
\[
    \int_{\lambda \Omega}\big \lvert \nabla [\Psi _{\lambda, r}(y)](x)\big \rvert^2 \text{d} x=\int_{B_{\lambda r}(y)}\big \lvert \nabla \mathfrak w_{B_{\lambda r}}^*(\abs{x-y})\big \rvert ^2\text{d} x=\int_{B_{\lambda r}}\abs{\nabla \mathfrak w_{B_{\lambda r}}^*(\xi)}^2\text{d} \xi.
\]
The nonlinear terms are such that
\[
    \int_{\lambda \Omega}\big \lvert [\Psi _{\lambda, r}(y)](x) \big \rvert ^p \text{d} x=\int_{B_{\lambda r}(y)}\big \lvert \mathfrak w_{B_{\lambda r}}^*(\abs{x-y}) \big \rvert^p \text{d} x=\int_{B_{\lambda r}}\abs{\mathfrak w_{B_{\lambda r}}^*(\xi)}^p \text{d} \xi.
\]
With respect to the nonlocal nonlinearities, first note that
\begin{multline*}
    \int_{\lambda \Omega \times \lambda \Omega} \frac{[\Psi _{\lambda, r}(y)]^2(x)[\Psi _{\lambda, r}(y)]^2(z)}{\abs{x-z}}\text{d} x\text{d} z=\\ \int_{B_{\lambda r}(y) \times B_{\lambda r}(y)} \frac{[\mathfrak w_{B_{\lambda r}}^*(\abs{x-y})]^2[\mathfrak w_{B_{\lambda r}}^*(\abs{z-y})]^2}{\abs{x-z}}\text{d} x\text{d} z
    =\int_{B_{\lambda r} \times B_{\lambda r}} \frac{[\mathfrak w_{B_{\lambda r}}^*(\xi)]^2[\mathfrak w_{B_{\lambda r}}^*(\zeta)]^2}{\abs{\xi-\zeta}}\text{d} \xi \text{d} \zeta.
\end{multline*}
Then, since
\[
    (x,z)=(\xi+y,\zeta+y)\in B_{\lambda r}(y)\times B_{\lambda r}(y)\quad \Longleftrightarrow \quad (\xi,\zeta)\in B_{\lambda r}\times B_{\lambda r},
\]
it follows that the terms involving the smooth parts of the Green functions satisfy
    \begin{align*}
        \sigma_{\lambda,r}(y)&:=\int_{\lambda \Omega \times \lambda \Omega}H(x,z;\lambda \Omega)[\Psi _{\lambda, r}(y)]^2(x)[\Psi _{\lambda, r}(y)]^2(z)\text{d} x\text{d} z\\ 
        &=\int_{B_{\lambda r}\times B_{\lambda r}} H(\xi+y,\zeta+y; \lambda \Omega)[\mathfrak w_{B_{\lambda r}}^*(\xi)]^2[\mathfrak w_{B_{\lambda r}}^*(\zeta)]^2\text{d} \xi \text{d} \zeta.
    \end{align*}
Summing up, we can write
\[
    I(\Psi _{\lambda, r}(y);\lambda \Omega)=I(\mathfrak w_{B_{\lambda r}}^*; \mathbb R^3)-\frac{\sigma_{\lambda, r}(y)}{4}.
\]
But then, by using \eqref{relacion} (first with $D=\lambda \Omega$, $u=\Psi _{\lambda, r}(y)$, then with $D=B_{\lambda r}$, $u=\mathfrak w_{B_{\lambda r}}^*$), \eqref{estimativa} with $D=B_r$, we find that
\begin{align*}
    I(\Psi _{\lambda, r}(y);\lambda \Omega)&<I(\mathfrak w_{B_{\lambda r}}^*;\mathbb R^3)\\
    &=I(\mathfrak w_{B_{\lambda r}}^*; B_{\lambda r})+\frac{1}{4}\int_{B_{\lambda r}\times B_{\lambda r}} H(x,z; B_{\lambda r})[\mathfrak w_{B_{\lambda r}}^*(x)]^2[\mathfrak w_{B_{\lambda r}}^*(z)]^2\text{d} x\text{d} z\\
    &=b^*_\lambda+\frac{1}{4}\int_{B_{\lambda r}\times B_{\lambda r}} H(x,z; B_{\lambda r})[\mathfrak w_{B_{\lambda r}}^*(x)]^2[\mathfrak w_{B_{\lambda r}}^*(z)]^2\text{d} x\text{d} z\\
    &<b^*_\lambda+\frac{M_{B_r}\rho^4}{4\lambda}\\
    &<l(\lambda).
\end{align*}
Taking $\lambda \geq \Lambda$, by Proposition \ref{propII_4.1}, it follows that $\beta (\Psi _{\lambda, r}(y) )=y$.

\medskip

Therefore, for $\lambda\geq \Lambda$, we have the following diagram of continuous maps
\[
    \lambda\Omega_{r}^{-}\xrightarrow{\ \Psi _{\lambda, r} \ } [I(\cdot\,;\lambda \Omega)]^{l(\lambda)}\xrightarrow{\ \ \beta \ \ } \lambda \Omega_{r}^{+}\simeq \lambda\Omega_{r}^{-}
\]
and the composition is homotopic to the identity map of $\lambda\Omega_{r}^{-}$.

\medskip

With the above ingredients in hands the next result is standard,
but we present the proof for the reader's convenience.

\begin{prop}\label{propII_5.1}
        Let $2<p<3$ and  $\rho \in (0,\rho_1)$. Then, for any $\lambda \geq \Lambda$
(the one given in  Proposition \ref{propII_4.1}) it holds
    \[
        \emph{cat}\, [I(\cdot \, ; \lambda \Omega)]^{l(\lambda)} \geq \emph{cat}\, \lambda\Omega=\emph{cat}\, \Omega.
    \]
\end{prop}
\begin{proof}
    Assume that $\cat\, [I(\cdot \, ; \lambda \Omega)]^{l(\lambda)}=n$; thus
\[
    [I(\cdot \, ; \lambda \Omega)]^{l(\lambda)}=F_1 \cup \ldots \cup F_n,
\]
where each  $F_i$ is closed and contractible in $[I(\cdot \, ; \lambda \Omega)]^{l(\lambda)}$. Hence, for each $i\in \{1,\ldots, n\}$, there exists 
\[
    h_i \in C \Big( [0,1]\times F_i, [I(\cdot \, ; \lambda \Omega)]^{l(\lambda)} \Big)
\]
with
    \[
        h_i(0,u)=u\quad \text{and}\quad h_i(1,u)=w_i \in [I(\cdot \, ; \lambda \Omega)]^{l(\lambda)}, \quad \forall \, u\in F_i\,.
    \]  
    Note that the sets $K_i:= \Psi ^{-1}_{\lambda, r}(F_i)$ are closed and satisfy   
    \[
        \lambda \Omega^-_r =K_1 \cup \ldots \cup K_n.
    \]   
Now we claim that each $K_i$ is contractible in $\lambda \Omega^+_r$. Indeed, for any $i\in \{1,\ldots, n\}$ fixed, the map 
    \[
    g_i: [0,1]\times K_i \to \lambda \Omega^+_r \quad \text{defined by}\quad g_i(t,y):= \beta\big(h_i\left(t,\Psi _{\lambda, r}(y)\right)\big) 
    \]
    is continuous and such that, for all $y\in K_i\,$,
\[
\begin{array}{ll}
        g_i(0,y)=\beta \left(\Psi _{\lambda, r}(y) \right)=y, \medskip \\
         g_i(1,y)=\beta(w_i) \in \lambda \Omega^+_r \quad \ \text{ (by Proposition \ref{propII_4.1})}.
\end{array}
\]  
   Then, by \eqref{II_4.9},
    \[
        \textrm{cat}\, \Omega = \textrm{cat}_{\lambda \Omega^+_r}(\lambda \Omega^-_r)\leq n=\textrm{cat}\, [I(\cdot \, ; \lambda \Omega)]^{l(\lambda)}.
    \]
This concludes the proof.    
\end{proof}

\medskip

\section{Proof of Theorem \ref{th:main}} 
\label{sec:finale} 

\medskip

Let us recall here a compactness condition useful to implement variational methods. In general, if $H$ is a Hilbert space, $\mathcal M\subset H$ a submanifold and $I:H\to \mathbb R$ a $C^{1}$ functional, we say that $I$ satisfies the Palais-Smale condition on $\mathcal M$ 
if any sequence $\{u_{n}\}\subset \mathcal M$ such that
$$ \{I(u_{n}) \}  \  \text{ is convergent  and }   (I_{|\mathcal M})'(u_{n}) \to 0$$
possesses a subsequence converging to some $u \in \mathcal M$. 
We will also say that $I$ constrained to $\mathcal M$ satisfies the $(PS)$ condition.

It is known that our functional $I(\cdot \, ; \lambda \Omega)$ constrained to $\mathcal M_\rho( \lambda \Omega)$ satisfies the $(PS)$ condition  (see  \cite[Appendix]{PS}); hence by the Ljusternik-Schnirelmann theory and Proposition \ref{propII_5.1}, for a fixed $\rho\in(0,\rho_1)$ and for every $\lambda \geq \Lambda=\Lambda(\rho)$, the functional $I(\cdot \, ; \lambda \Omega)$ constrained to $\mathcal M_\rho( \lambda \Omega)$ has at least $N:=\textrm{cat}\, (\lambda \Omega) =\textrm{cat}\, \Omega$ distinct critical points $\{u_{\rho,\lambda}^{i}\}_{i=1,\ldots,N}$ with energy satisfying 
\begin{equation}\label{eq:comparacion}
c_{\lambda}\leq I(u^{i}_{\rho,\lambda};\lambda \Omega)\leq b_\lambda^*+\frac{1}{\lambda}\bigg(1+\frac{M_{B_r}\rho^4}{4}\bigg),\quad i=1,\ldots, N,
\end{equation}
and the right hand side can be made less than $c_\infty/2$ by Proposition \ref{propII_3.1}, up to taking a greater value of $\Lambda$.
This means that $\sup_{\lambda \geq \Lambda}I(u^{i}_{\rho,\lambda};\lambda \Omega)<0$. The Lagrange multipliers $\{\omega^{i}_{\rho,\lambda}\}_{i=1,\ldots, N}\subset \mathbb R$ are associated to the critical points.

\medskip

 Moreover by \eqref{eq:comparacion}, Proposition \ref{propII_3.1} and \eqref{estimativa} (with $D=\Omega$) it is, for each $i=1,\ldots, N$,
    \[
        I(u^{i}_{\rho,\lambda};\lambda \Omega)=I(u^{i}_{\rho,\lambda};\mathbb R^{3})-\frac{1}{4}\int_{\lambda \Omega\times \lambda \Omega} H(x,y; \lambda \Omega)[u^{i}_{\rho,\lambda}]^2(x)[u^{i}_{\rho,\lambda}]^2(y)\text{d} x\text{d} y \to c_{\infty}
    \]
as $\lambda\to+\infty$ so that, in particular, each family
    $\{u^{i}_{\rho,\lambda}\}_{\lambda \geq \Lambda}$ (extended by zero outside of $\lambda \Omega$)
    provides a minimizing sequence for $c_\infty$ as $\lambda\to +\infty$. Therefore, up to translations (recall Fact 2 in Section \ref{sec:prelim}), for each $i\in \{1,\ldots, N\}$,
    \[
        u^{i}_{\rho,\lambda}\to \mathfrak w_\infty>0 \text{ in } H^1(\mathbb R^3) \quad \text{as } \lambda\to+\infty.
    \]
Using \eqref{lagrange:infinito}, the solutions also satisfy
    \begin{align*}
        \rho^2 \omega^{i}_{\rho,\lambda}&= \int_{\lambda \Omega} \abs{\nabla u^{i}_{\rho,\lambda}}^2 \text{d} x +\int_{\lambda \Omega}\phi_{u^{i}_{\rho,\lambda}} [u^{i}_{\rho,\lambda}]^2 \, \text{d} x- \int_{\lambda \Omega} \abs{u^{i}_{\rho,\lambda}}^p \text{d} x\\
        &= \int_{\mathbb R^3} \abs{\nabla \mathfrak w_\infty}^2 \text{d} x +\int_{\mathbb R^3}\varphi _{\infty} \mathfrak w_\infty^2 \, \text{d} x- \int_{\mathbb R^3} \mathfrak w_\infty^p \text{d} x+o_\lambda(1)\\
        &= \rho^2 \omega_\infty +o_\lambda(1),
    \end{align*}
    from which, for every $i\in \{1,\ldots,N\}$,
    $\lim_{\lambda \to +\infty}\omega^{i}_{\rho,\lambda}= \omega_\infty<0.$

\medskip

Defining, as usual, $u^+:=\max \{u,0\}$ we see that the whole previous analysis is valid for the functional
    \[
        I^+(u; \lambda\Omega):= \frac{1}{2}\int_{\lambda\Omega} \abs{\nabla u}^2 \text{d} x+\frac{1}{4}\int_{\lambda\Omega} \phi_u u^2 \text{d} x-\frac{1}{p}\int_{\lambda\Omega} (u^+)^p \text{d} x
    \]
restricted to $\mathcal M_{\rho}(\lambda \Omega)$. Then we have, for any fixed $\rho\in(0,\rho_1)$ and 
$\lambda \geq \Lambda$, at least $\cat(\lambda \Omega)=\cat\Omega$  solutions $(u_{\rho,\lambda}, \omega_{\rho,\lambda})$ of
 \begin{equation*}
 \left\{
	    \begin{array}{ll}
	-\Delta u  +\phi_{u} u - (u^{+})^{p-1}= \omega u& \quad \text{in }\lambda\Omega, \smallskip\\
	u =0 &\quad \text{on }\partial (\lambda\Omega),
	    \end{array}
    \right.
\end{equation*}
with $u_{\rho,\lambda}$ nonnegative, $\displaystyle\int_{\lambda\Omega} u_{\rho,\lambda}^{2}\text{d} x=\rho^2$ and
$\omega_{\rho,\lambda}<0$. The maximum principle allows to conclude that $u_{\rho,\lambda}>0$ in $\Omega$.
 
The final part of  Theorem \ref{th:main} is proved 
by using the same ideas of \cite{BCP} if $\Omega$ (and hence $\lambda\Omega$) is not contractible in itself. Indeed, in this case the compact set $K:=\overline{{\Psi_{\lambda,r}(\lambda \Omega_{r}^{-})}}\subset \mathcal M_{\rho}(\lambda\Omega)$
can not be contractible in the sublevel $[I(\cdot\, ; \lambda\Omega)]^{l(\lambda)}$. Take now $\widehat{u} \in \mathcal M_{\rho}(\lambda\Omega)\setminus K$ such that 
$\widehat{u}\geq 0$; thus $\widehat{u}\not \equiv 0$. Then
$I(\widehat{u}; \lambda\Omega)>l(\lambda),$
and the cone
$$\mathfrak C:=\Big\{t\widehat{u}+(1-t)u: t\in [0,1], u\in K \Big\}$$
does not contain the zero function (observe that the functions in $K$  are nonnegative).
Then, the projection of the cone on  $\mathcal M_{\rho}(\lambda\Omega)$
\[
    \mathrm P(\mathfrak C):=\bigg\{ \rho\frac{w}{|w|_{L^{2}(\lambda\Omega)}}: w\in \mathfrak C\bigg\}\subset \mathcal M_{\rho}(\lambda\Omega),
\]
is well defined. Let
\[
m(\lambda):=\max_{\mathrm P(\mathfrak C)}I(\cdot\, ;\lambda\Omega)>l(\lambda).
\]
Since $ K\subset \mathrm P(\mathfrak C)\subset \mathcal M_{\rho}(\lambda\Omega)$
and $\mathrm P(\mathfrak C)$ is contractible in $ [I(\cdot\,; \lambda\Omega)]^{m(\lambda)}$,
we infer that $K$ is also contractible in $ [I(\cdot\,; \lambda\Omega)]^{m(\lambda)}$.
On the other hand, recalling that $K$  is 
 not contractible  in $[I(\cdot\,;\lambda\Omega)]^{l(\lambda)}$ and  
the Palais-Smale condition is satisfied, we conclude that there is another nonnegative critical point 
$\widetilde u_{\rho,\lambda}$  for the functional
with energy level between $l(\lambda)$ and $m(\lambda)$; hence another Lagrange multiplier $\widetilde \omega_{\rho,\lambda}$, which is associated to $\widetilde u_{\rho,\lambda}$. In other words, there is  another solution $(\widetilde u_{\rho,\lambda}, \widetilde \omega_{\rho,\lambda})$ to the problem with $\widetilde u_{\rho,\lambda}$ nonnegative. However, we cannot guarantee as before the positivity of this additional solution $\widetilde u_{\rho,\lambda}$, since we do not know if it is at a nonpositive level of the functional and if the associated Lagrange multiplier $\widetilde \omega_{\rho,\lambda}$ is negative.

\section{Appendix}
The aim of this section is twofold.

\begin{itemize}
\item[i)]  On one hand, we show the strike difference between the case of the whole space $\mathbb R^{3}$ and the situation for a bounded 
domain $D$, from the point of view of the minimum of the functional with fixed $L^{2}$-norm
equal to $\rho$. In fact, while in $\mathbb R^{3}$ the minimum of the constrained functional is negative
(see Fact 2) for any $\rho>0$, in the bounded domain the minimum is positive
for  $\rho\leq \rho_{D}$, a constant  depending on the same domain. 

\item[ii)] On the other hand, we show that if the domain is expanding (namely we consider $\lambda D$ with $\lambda\to+\infty$), then $\rho_{\lambda D}\to 0$. As a consequence, when $\rho$ is fixed, if the domain  is expanding we will have $\rho_{\lambda D} <\rho$ and then the computations which led to get  the positivity of the minimum are not allowed.
\end{itemize}

\medskip

To address the first issue, let $D$ be a smooth bonded domain. Let us start by recalling that
\begin{lemma}(see \cite[Lemma 3.1 and Remark 3.2]{TMNA}) If $D\subset \mathbb R^{3}$ is a smooth bounded domain, $p\in(2,6)$ and $r\in(0,p)$, then 
for every $u\in H^{1}(D)$ it is 
$$|u|^{p}_{L^{p}(D)} \leq C_{D}^{p-r}\|u\|_{H^{1}(D)}^{p-r}|u|_{L^{2}(D)}^{r}$$
where $$C_{D} = \sup_{u\in H^{1}(D)\setminus\{0\} } \frac{ |u|_{L^{\frac{2(p-r)}{2-r}} (D)} }{\|u\|_{H^{1}(D)}}>0.$$
In particular, for every $u\in H^{1}_{0}(D)$ it holds
$$|u|^{p}_{L^{p}(D)} \leq C_{D}^{p-r} \left( 1+\frac{1}{\mu_{1}(D)}\right)^{\frac{p-r}{2}}|\nabla u|_{L^{2}(D)}^{p-r}|u|_{L^{2}(D)}^{r},$$
where $\mu_{1}(D)$ is the first eigenvalue of the Laplacian in $H^{1}_{0}(D)$.
\end{lemma}
By choosing $r=p-2$ the Lemma furnishes, for $u\in \mathcal M_{\rho}(D)$, 
\begin{equation}\label{eq:llave}
|u|^{p}_{L^{p}(D)} \leq C_{D}^2 \left( 1+\frac{1}{\mu_{1}(D)}\right)|\nabla u|_{L^{2}(D)}^{2}|u|_{L^{2}(D)}^{p-2} =: \widetilde C_{D}|\nabla u|_{L^{2}(D)}^{2}\rho^{p-2}.
\end{equation}
From  \eqref{eq:llave} it follows that
for a suitable small $\rho$, for example 
\begin{equation}\label{eq:rho}
\rho\leq\left( \frac{p}{4\widetilde C_{D}}\right)^{1/(p-2)}=:\rho_{D},
\end{equation}
it is $|u|^{p}_{L^{p}(D)} \leq \frac{p}{4}|\nabla u|^{2}_{L^{2}(D)}$ and therefore
\begin{equation}\label{positiva}
I(u; D)>\frac{1}{2}|\nabla u|_{L^{2}(D)}^{2}-\frac{1}{p}|u|^{p}_{L^{p}(D)} \geq \frac{1}{2}|\nabla u|_{L^{2}(D)}^{2} -  \frac{1}{4}|\nabla u|_{L^{2}(D)}^{2}\geq\frac{\rho^2}{4 }\mu_{1}(D)>0.
\end{equation}
The constant $\rho_{D}$ is not optimal; however, the above arguments show the following result that we were not able to find in the literature.
\begin{prop}\label{proposicion1}
Let  $D\subset \mathbb R^{3}$ be a bounded and smooth domain.  There is a constant $\rho_{D}>0$ such that,
for any  $\rho\in(0, \rho_{D}]$, it is
\[
    \min_{u\in \mathcal M_{\rho}(D)} I(u;D)>0.
\]
\end{prop}
The fact that the infimum is achieved is standard and follows since we have compactness being in a bounded domain.

\medskip

We stress that in the above result the value $\rho_{D}$ depends on the domain. Furthermore,  the statement of  Proposition \ref{proposicion1} does not hold when we fix the $L^2$-norm a priori and consider the minimum of the constrained functional  on a family of expanding domains. 

To see this,
let $D$ be a smooth bounded domain, and fix $\rho\in(0,\rho_D]$.
    We have the following
\medskip

   {\bf Claim:} with the above notations
\begin{equation*}\label{eq:diverging}
        \{\widetilde C_{\lambda D}\}_{\lambda>1}\quad \text{is diverging as}\quad \lambda \to +\infty.
\end{equation*}

   Assuming the Claim  for a moment, since $\rho$ is fixed and $\lambda$ is increasing, we get that $\rho_{\lambda D} < \rho$, for a suitable large $\lambda$. Thus \eqref{eq:rho} is violated and we cannot guarantee anymore neither the second inequality in  \eqref{positiva} nor the subsequent positivity of the infimum on an expanding domain. 

In fact, we have seen in Proposition \ref{propII_3.1} that with $\rho$ fixed, considering the expanding domains
$B_{\lambda r}$ and $\lambda\Omega$,  the infimum of the functional is negative for large values of $\lambda$.
 It is not surprising then, since we can see $\mathbb R^{3}$ in some sense  as the limit of $B_{\lambda r}$ as $\lambda\to +\infty$, that $c_{\infty}<0$ (see Fact 2).
    
    \medskip
    
    Let us show the Claim; in first place, taking any $u\in H^1_0(D)\setminus\{0\}$ and $\lambda>1$, the function defined by
    \[
        x\in \lambda D \quad \longmapsto \quad v(x):=\frac{1}{\lambda^{3/2}}u(x/\lambda) \in \mathbb R
    \]
    is such that $v\in H^1_0(\lambda D)$. Since
    \[
        \int_{\lambda D}v^2 \text{d} x = \int_D u^2 \text{d} x \quad \text{and} \quad \int_{\lambda D}\abs{\nabla v}^2 \text{d} x = \frac{1}{\lambda^2}\int_{D}\abs{\nabla u}^2 \text{d} x,
    \]
    we have
    \[
        \mu_1(\lambda D) \leq \frac{\norm{v}^2_{H^1_0(\lambda D)}}{\abs{v}^2_{L^2(\lambda D)}}=\frac{1}{\lambda^2}\frac{\norm{u}^2_{H^1_0(D)}}{\abs{u}^2_{L^2(D)}}.
    \]
    By the arbitrariness of $u$, this implies that
    \begin{equation}\label{eq:primautoval}
      \mu_1(\lambda D) \leq \frac{1}{\lambda^2} \mu_1(D).
     \end{equation}
      
    On the other hand, let us choose a point $x_0$ in $D$ and $\delta>0$ such that $B_\delta(x_0)\subset D$. Also we fix a function  $w\in C^{\infty}_0(B_\delta(x_0))$  such that $w>0$ in $B_\delta(x_0)$.
    For any $\lambda>1$, $B_\delta(\lambda x_0) \subset B_{\lambda \delta}(\lambda x_0) \subset \lambda D$, and the function
    \[
        x\in B_{\delta}(\lambda x_0) \quad \longmapsto \quad w_\lambda (x):=w\left( x-(\lambda-1)x_0 \right) \in \mathbb R
    \]
    is such that
    \[
        \int_{B_{\delta}(\lambda x_0)} w_\lambda^p \text{d} x= \int_{B_{\delta}( x_0)} w^p \text{d} x, \quad \int_{B_{\delta}(\lambda x_0)} w_\lambda^2 \text{d} x = \int_{B_{\delta}( x_0)} w^2 \text{d} x,
    \]
    and
    \[
    \int_{B_{\delta}(\lambda x_0)} \abs{\nabla w_\lambda}^2 \text{d} x= \int_{B_{\delta}(x_0)} \abs{\nabla w}^2 \text{d} x.  
    \]
    Hence $w_\lambda \in H^1(\lambda D)$, and (being $2(p-r)/ (2-r)= 4/(4-p)\in(2,6)$)
    \[
        0<\frac{\abs{w}_{L^{4/(4-p)}(B_\delta(x_0))}}{\norm{w}_{H^1(B_\delta(x_0))}} = \frac{\abs{w_\lambda}_{L^{4/(4-p)}(B_\delta(\lambda x_0))}}{\norm{w_\lambda}_{H^1(B_\delta(\lambda x_0))}} \leq \sup_{u\in H^1(\lambda D)\setminus \{0\}} \frac{\abs{u}_{L^{4/(4-p)}(\lambda D)}}{\norm{u}_{H^1(\lambda D)}} =C_{\lambda D},
    \]
where the expression in the right hand side of the last inequality is the optimal constant in the embedding
    \(
        H^1(\lambda D) \hookrightarrow L^{4/(4-p)}(\lambda D).
    \)
    But then, recalling the definition of $\widetilde C_{\lambda D}$ and  \eqref{eq:primautoval} we find that
    $$\lim_{\lambda\to+\infty} \widetilde C_{\lambda D} = \lim_{\lambda\to+\infty} C_{\lambda D}^2 \left( 1+\frac{1}{\mu_{1}(\lambda D)}\right) \geq 
    \lim_{\lambda\to+\infty} C_{\lambda D}^2 \left( 1+\frac{\lambda^{2}}{\mu_{1}( D)}\right) =+\infty
    $$
    which proves the Claim.

\bigskip

{\bf Acknowledgements.}

Edwin G. Murcia is supported by Departamento de Matem\'aticas, Pontificia Universidad Javeriana, through research project ID PRY 000000000011253.

G. Siciliano was supported by Capes, CNPq, FAPDF Edital 04/2021 - Demanda Espont\^anea, 
Fapesp grants no. 2022/16407-1 and 2022/16097-2 (Brazil),  
PRIN PNRR, P2022YFAJH “Linear and Nonlinear PDEs: New directions and applications”, and 
INdAM-GNAMPA project n. E5324001950001 ``Critical and limiting phenomena in nonlinear elliptic systems'' (Italy).

\medskip

{\bf Data availability.} Data sharing not applicable to this article as no datasets were generated or analysed during the current study.

\medskip

{\bf Conflict of interest.} The authors declare that they have no conflict of interest.

\medskip

{\bf Author contributions.} The authors contributed equally to the writing of this article. All authors
read and approved the final manuscript.

\end{document}